\documentclass[12 pt]{amsart}
\usepackage{amsmath,amssymb,amsthm}
\usepackage{longtable,}
\usepackage{a4wide}
\usepackage{color}

\usepackage{tikz}
\usetikzlibrary{decorations.pathreplacing}

\newtheorem{thm}{Theorem}[section]
\newtheorem*{thm*}{Theorem}
\newtheorem*{conj*}{Conjecture}

\newtheorem{lem}[thm]{Lemma}
\newtheorem{prop}[thm]{Proposition}

\theoremstyle{remark}
\newtheorem{rem}[thm]{Remark}

\theoremstyle{definition}

\newcounter{claim}[thm]

\newcommand{\sym}[1]{\mathrm{Sym}(#1)}

\newcommand{\dih}[1]{\mathrm{Dih}(#1)}
\newcommand{\grp}[1]{\langle #1 \rangle}				
\newcommand{\op}[2]{\mathrm{O}_{#1}(#2)}	
\newcommand{\norm}[2]{\text{N}_{#1}(#2)}				
\newcommand{\cent}[2]{\text{C}_{#1}(#2)}				
\newcommand{\vst}[2]{G_{#1}^{[ #2]}}		
\newcommand{\vstx}[1]{G_x^{[ #1 ]}}			
\newcommand{\vsty}[1]{G_y^{[ #1 ]}}			
\newcommand{\syl}[2]{\mathrm{Syl}_{#1}( #2)}	
\newcommand{\zent}[1]{\mathrm{Z}(#1)}
\newcommand{\ba}[1]{\overline{#1}}

\title{A class of semiprimitive groups that are graph-restrictive\footnote{ T\lowercase{his research forms part of an  \uppercase{A}ustralian \uppercase{R}esearch \uppercase{C}ouncil \uppercase{D}iscovery \uppercase{P}roject (project number \uppercase{DP}120100446).}} }

\author{Michael Giudici}
\author{Luke Morgan}

\address{
School of Mathematics and Statistics (M019)\\
University of Western Australia\\
Crawley, 6009\\
Australia} 

\email{michael.giudici@uwa.edu.au}	\email{luke.morgan@uwa.edu.au}

\begin{document}

\begin{abstract}
We prove that an infinite family of semiprimitive groups are graph-restrictive. This adds to the  evidence for the validity of the PSV Conjecture and increases the minimal imprimitive degree for which this conjecture is open to 12. 
 Our result can be seen as a generalisation of the well-known theorem of Tutte on cubic graphs. The proof uses the amalgam method, adapted to this new situation. 
\end{abstract}

\maketitle

\section{Introduction}
A famous theorem of Tutte says that  in every group that acts faithfully and arc-transitively on a finite connected cubic graph, the order of a vertex stabiliser divides 48. We shall prove a generalisation of Tutte's theorem which holds for infinite families of graphs with valencies powers of three. 
To state our result, we must introduct some terminology. For a vertex $x$ of a graph $\Gamma$ we write $\Gamma(x)$ for the set of neighbours of $x$ and $G_x^{\Gamma(x)}$ for the permutation group induced by the stabiliser $G_x$ of $x$ on $\Gamma(x)$. 
Let $L$ be a permutation group and let $\Gamma$ be a connected graph with vertex-transitive group of automorphisms $G$ such that $|G_x|$ is finite for all vertices $x$.  If for  $x\in \Gamma$ there is a permutation isomorphism $G_x^{\Gamma(x)} \cong L$  we say that the pair $(\Gamma,G)$ is \emph{locally} $L$. Following  \cite{verret},  if there is a constant $C$ such that for every locally $L$ pair $(\Gamma,G)$ we have $|G_x| \leqslant  C$, we say that $L$ is \emph{graph-restrictive}. Our generalisation is the following.

\begin{thm}
\label{mainthm}
Let  $L$ be the Frobenius group of order $2\cdot3^n$ and degree $3^n$ with elementary abelian Frobenius kernel. Then $L$ is graph-restrictive.
\end{thm}

In the language introduced above, Tutte's seminal work on cubic arc-transitive graphs \cite{Tutte1,Tutte2}  asserts that $C_3$ and $\sym{3}$ are graph-restrictive (as subgroups of $\sym{3}$). In the same vein, Gardiner \cite{gard} showed that any transitive subgroup of $\sym{4}$ other than $\mathrm{Dih}(8)$ is graph-restrictive, while Sami \cite{sami} showed that any dihedral group of odd degree is as well. We extend our  terminology to describe these situations. Let $\mathcal P$ be a property of permutation groups (such as transitive) and let $\Gamma$ be a connected graph with vertex-transitive group of automorphisms $G$. We say that the pair $(\Gamma,G)$ is locally $\mathcal P$ if the permutation group $G_x^{\Gamma(x)}$ satisfies $\mathcal P$.

Motivated by the above results, particular attention has been paid to studying the structure of vertex stabilisers in locally primitive graphs. A transitive permutation group on a set $\Omega$ is called \emph{primitive} if it preserves no nontrivial partitions of $\Omega$.  A difficult and long-standing conjecture of Weiss  \cite{weisscon} asserts that  primitive permutation groups are graph-restrictive.
It follows from work of Trofimov and Weiss (see \cite[Theorem 1.4]{trofweiss1}) that the Weiss Conjecture is true for the subclass of 2-transitive groups. This required a deep theorem of Trofimov (see \cite{trof1}, \cite{trof2}, \cite{trof3,trof4} and \cite{trof5,trof6,trof7,trof8}) dealing with the case of permutation groups with socle $\mathrm{PSL}_n(q)$ in a doubly transitive action. In joint work Trofimov and Weiss \cite{trofweiss1,trofweiss2} have shown that the Weiss Conjecture holds for permutation groups with socle $\mathrm{PSL}_n(q)$ acting on subspaces.
Praeger, Spiga and Verret \cite{PrSV} reduced the truth of the Weiss Conjecture to a question about simple groups and Praeger, Pyber, Spiga and Szab\'o \cite{PPSS} subsequently showed that it is true for locally primitive pairs $(\Gamma,G)$ where the composition factors of $G$ have bounded rank.

A natural question is the following: what is the correct context  for graph-restrictive permutation groups and the Weiss Conjecture?
A permutation group is called \emph{quasiprimitive} if every nontrivial normal subgroup is transitive.  This class of groups includes all primitive groups and Praeger generalised the Weiss conjecture to quasiprimitive permutation groups \cite{praeger}.  Poto\v{c}nik, Spiga and Verret \cite{PSV} have shown that if a transitive permutation group  $L$ is graph-restrictive then it is \emph{semiprimitive}, that is, every normal subgroup of $L$ is either transitive or semiregular. This led the authors to make the PSV Conjecture:  a transitive permutation group is graph-restrictive if and only if it is semiprimitive. In later work of the second and third authors it  was shown  that intransitive permutation groups are graph-restrictive  if and only if they are semiregular \cite{pabloverret}.

The class of semiprimitive groups includes the classes of primitive, quasiprimitive and Frobenius groups. Recently this class was studied by Bereczky and Mar\'oti \cite{BM} due to a connection with collapsing monoids. We note that our definition follows \cite{PSV} (as opposed to \cite{BM}) by including all regular groups as semiprimitive groups.  There is an easy argument to show that regular groups are graph-restrictive. As part of their investigations, Poto\v{c}nik, Spiga and Verret showed that the action of $\mathrm{GL}_{2}(p)$ on the set of nontrivial vectors of a 2-dimensional vector space over $\mathbb F_p$ are graph-restrictive. They also showed that all semiprimitive groups of degree at most 8 are graph-restrictive.  One group whose status was left open in \cite{PSV} is the Frobenius group $3^2{:}2$ acting on 9 points (which fails to be quasiprimitive).  Investigating this group is a natural extension of the work of Tutte on the group $\sym{3}=3{:}2$ and  was seen as a natural test case for how methods used for primitive groups could be extended to semiprimitive groups. Our theorem therefore deals with the least degree imprimitive open case listed in \cite{PSV} and pushes the degree of the smallest imprimitive unknown case to 12 (although there is still one primitive group of degree 9 and one of degree 10 listed in \cite{PSV} for which the conjecture is unknown).

One important tool when studying the structure of vertex-stabilisers in vertex-transitive graphs is the so-called Thompson-Wielandt Theorem. First  repurposed for the locally primitive situation by Gardiner \cite[(2.3)]{gard}, a  generalisation  due to Spiga \cite{spigath-w}  shows that in a locally semiprimitive pair $(\Gamma,G)$  there exists a prime $p$ such that $G_{uv}^{[1]}$ is a $p$-group for each arc  $(u,v)$ of $\Gamma$. Another important tool has been the amalgam method which goes back to Goldschmidt \cite{gold}. Since this there have been many successful applications, notably in \cite{stelldelgado} where the method was refined. In contrast to our setting, for most of the cases considered in \cite{stelldelgado} the  action is locally 2-transitive,
 therefore  rather stronger than locally semiprimitive. The proof of our theorem relies upon extending the amalgam method to the semiprimitive group $3^n{:}2$. What enables us to make progress in this new situation is a consideration of block systems for the local action. This, together with a modification of an an argument of Weiss \cite{weissgold}, presented in Lemma~\ref{lem:boundonbisgood}, shows that a bound on the so-called critical distance in the amalgam method will show that the group $L$ is graph-restrictive. For an introduction to the amalgam method we refer the reader to \cite[10.3]{kurzweilstellmacher}.\\ \\

\textbf{Acknowledgements:} The authors thank Gabriel Verret and Pablo Spiga for reading an early version of the paper and providing useful comments.

\section{Preliminaries}
Let $n\in \mathbb N$ let $L\cong 3^n{:}\mathrm C_2$ be as in Theorem~\ref{mainthm}. Assume that $n\geqslant 2$ so that  $L$ acts faithfully and imprimitively on the $n$-dimensional vector space $V$ over $\mathbb F_3$. All proper normal subgroups of $L$ are elementary abelian of order $3^m$ for some $m$ and so are semiregular. Thus $L$ is semiprimitive. Blocks of imprimitivity for $L$ on $V$ arise from translates of subspaces. 
%
Note that the point stabilisers are the Sylow 2-subgroups of $L$ and the stabiliser of a block of size three is isomorphic to $\sym{3}$.

\begin{prop}
\label{genofx}
For  each basis  $\mathcal B$ of $V$ we have $L=\grp{L_0, L_w \mid w \in \mathcal B}$.
\end{prop}
\begin{proof}
For $w\in \mathcal B$ we note that $\grp{L_0,L_w}=L_{\grp{w}}$, the stabiliser of the block $\{0,w,-w\}$, and so is isomorphic to $\sym{3}$. The translation by $w$ is thus contained in $L_{\grp{w}}$, and therefore 
$$\grp{L_0, L_w \mid w \in \mathcal B}$$
contains the Sylow 3-subgroup of $L$ and we are done.
\end{proof}

Now we assemble some facts concerning $\mathbb F_2$-modules for $L$.

\begin{prop}\label{p:Xmodules}
There is a bijection between the sets of irreducible $\mathbb F_2$-modules for $L$ and for $\op{3}{L}$. Moreover, the nontrivial irreducible $\mathbb F_2$-modules for $L$ have dimension two and if $t\in \op{3}{L}$ and $V$ is a non-trivial irreducible $\mathbb F_2L$-module, either $[t,V]=1$ or $t$ is fixed point free on $V$.
\end{prop}
\begin{proof}
Observe that there are the same number of modules in each set and each irreducible $\mathbb F_2L$-module restricts to an irreducible $\mathbb F_2\op{3}{L}$-module. The nontrivial irreducible modules all have dimension 2, therefore the moreover part of the proposition follows.
\end{proof}

The following can be found as part of \cite[8.2.7 and 8.4.2]{kurzweilstellmacher}.

\begin{lem}[Coprime Action]
Let $G$ and $V$ be finite groups. Suppose that $G$ acts on $V$ and $(|G|,|V|)=1$. Then the following hold.
\begin{itemize}
\item[(a)] $[V,G,G] = [V,G]$.
\item[(b)] If $V$ is abelian then $V=\cent{V}{G} \times [V,G]$.
\end{itemize}
\end{lem}

\section{Amalgam method}

Let $L$ be the group defined in Theorem~\ref{mainthm} and assume that $\Gamma$ is a connected graph with an arc-transitive group $G$ of automorphisms so that $(\Gamma, G)$ is locally $L$. We may assume that $\Gamma$ is a tree \cite[Chapter 1, \S 4]{serre} and by Tutte's Theorem that $n\geqslant 2$. For adjacent vertices $x$ and $y$ of $\Gamma$,
  we define $G_x^{[1]}$ to be the kernel of the action of $G_x$ on $\Gamma(x)$ and 
   $$G_{xy}^{[1]}=\vstx{1} \cap \vsty{1},$$ the kernel of the action of $G_x \cap G_y$ on  $\Gamma(x)\cup\Gamma(y)$. 

 The following lemma is central to  arguments involving the amalgam method.
   \begin{lem}
   \label{lem:fundlem}
Let $e=\{x,y\}$ be an edge of $\Gamma$ and suppose that $K\leqslant  G_x \cap G_y$. If either (a) or (b) below hold, then $K=1$.
\begin{itemize}
\item[(a)] $\norm{G_x}{K}$ and $\norm{G_y}{K}$ are transitive on $\Gamma(x)$ and $\Gamma(y)$ respectively.
\item[(b)] $\norm{G_x}{K}$ is transitive on $\Gamma(x)$ and $\norm{G_e}{K} \not\leqslant G_x \cap G_y$.
\end{itemize}
\end{lem}
\begin{proof}
Suppose that (a) holds and set $H = \grp{ \norm{G_x}{K},\norm{G_y}{K}} \leqslant \norm{G}{K}$. Since $\Gamma$ is connected, $H$ acts edge-transitively. Let $u$ be any vertex of $\Gamma$ and let $v$ be adjacent to $u$. Then there exists $h\in H$ such that $\{x,y\}^h = \{u,v\}$. Now we obtain 
$$K = K^h \leqslant  (G_x \cap G_y)^h = G_u \cap G_ v \leqslant  G_{u}$$
whence $K$ fixes every vertex of $\Gamma$, and therefore $K=1$. The case that (b) holds is similar and is omitted.
\end{proof}

We fix an edge $e=\{x,y\}$. Note that by edge-transitivity  statements proved about the edge $e$ apply to arbitrary edges. We may assume $G_{xy}^{[1]} \neq 1$.

\begin{lem}
\label{lem:isom of ge/gxy1 and fact of gxy}
The group $G_e$ is a 2-group, $G_e/G_{xy}^{[1]} \cong \mathrm{Dih}(8)$ and $G_x^{[1]}G_y^{[1]}=G_x \cap G_y$.
\end{lem}
\begin{proof}
Since $\vst{xy}{1} \neq 1$ we have $\vst{x}{1} \neq 1$ and so  $G_x^{[1]}\neq G_y^{[1]}$ by Lemma~\ref{lem:fundlem}. Then as $|(G_x \cap G_y)/G_x^{[1]}|=2$ we have $G_x^{[1]}G_y^{[1]}=G_x \cap G_y$ and $(G_x \cap G_y)/G_{xy}^{[1]}$ is elementary abelian of order four. Moreover, there is $t \in G_e - G_x \cap G_y$ which swaps $G_x^{[1]}$ and $G_y^{[1]}$ and $|G_e/G_{xy}^{[1]}|=8$ hence $G_e/G_{xy}^{[1]} \cong \dih{8}$.

The previous paragraph shows that for any prime $r\neq 2$ we have $\syl{r}{G_{xy}^{[1]}}=\syl{r}{G_x^{[1]}}$. Taking $S\in \syl{r}{G_x^{[1]}}$ the Frattini argument gives $G_x=\norm{G_x}{S} G_x^{[1]}$ and $G_e = \norm{G_e}{S}G_{xy}^{[1]}$.  In particular $\norm{G_x}{S}$ is transitive on $\Gamma(x)$ and $\norm{G_e}{S} \not\leqslant \vst{xy}{1}$. Now Lemma~\ref{lem:fundlem} gives $S=1$.
\end{proof}

For an edge $\{u,w\}$ of $\Gamma$ we define the following subgroups:
\begin{eqnarray*}
Q_v &	= & 	G_v^{[1]},\\
Z_{uw} 	&=	& \Omega_1(Z(G_u \cap G_w)), \\
Z_u 	&=	& \langle Z_{uw} ^{G_u} \rangle.
\end{eqnarray*}

Recall that $\Omega_1(P) = \langle x \mid x^p =1 \rangle$ for a $p$-group $P$  and note that since $G_u \cap G_w$ is a 2-group we have $Z_{uw}\neq 1$.

\begin{prop}\label{p:zxcentral}
$Z_x \leqslant \Omega_1(\zent{Q_x})$.
\end{prop}
\begin{proof}
Since $\cent{G_x}{Q_x} \leqslant \cent{G_x}{\vst{xy}{1}}$ and $\vst{xy}{1}$ is a nontrivial normal subgroup of $G_e$, by Lemma~\ref{lem:fundlem}(b) we have that  $\cent{G_x}{Q_x}$ is intransitive on $\Gamma(x)$. Since $\cent{G_x}{Q_x} \trianglelefteq G_x$ it follows that $\cent{G_x}{Q_x}$ is semiregular on $\Gamma(x)$ and so $2\nmid|\cent{G_x}{Q_x}:\cent{Q_x}{Q_x}|$. In particular, $\zent{Q_x}$ is a Sylow 2-subgroup of $\cent{G_x}{Q_x}$. We have that $[Q_x,Z_{xy}]\leqslant [G_x \cap G_y,Z(G_x \cap G_y)]=1$ so that $Z_{xy}$ centralises $Q_x$ and therefore $Z_{xy} \leqslant \zent{Q_x}$.  By the definition of $Z_{xy}$ we get $Z_{xy} \leqslant \Omega_1(\zent{Q_x})$. Since $\Omega_1(\zent{Q_x})$ is a characteristic subgroup of $Q_x$ and $G_x$ normalises $Q_x$, we obtain the result.
\end{proof}

The following diagram helps keep track of the subgroups defined so far.

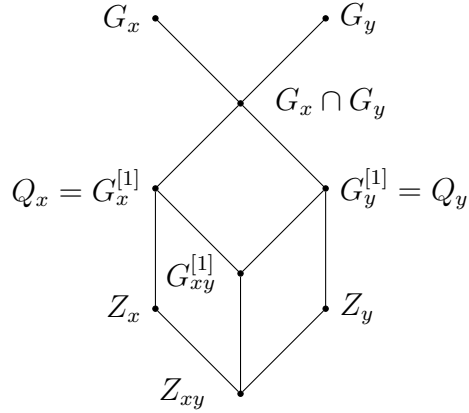
\begin{center}
\begin{tikzpicture}[scale=0.8]

\draw (0,0) node (p12) [label=right:${\ \ G_x \cap G_y}$,draw,circle,fill=black,minimum size=2pt,
                            inner sep=0pt] {}
-- ++(135:2.0cm) node (p1) [label=left:$G_x$,draw,circle,fill=black,minimum size=2pt,
                            inner sep=0pt] {};
\draw (p12) -- ++(45:2.0cm) node (p2) [label=right:$G_y$,draw,circle,fill=black,minimum size=2pt,
                            inner sep=0pt] {};
\draw (p12) -- ++(-45:2.0cm) node (gy1) [label=right:${ G_y^{[1]}=Q_y }$,draw,circle,fill=black,minimum size=2pt,
                            inner sep=0pt] {};
\draw (p12) -- ++(-135:2.0cm) node (gx1) [label=left:${ Q_x=G_x^{[1]}}$,draw,circle,fill=black,minimum size=2pt,
                            inner sep=0pt] {};
\draw (gx1) -- ++(-45:2.0cm) node (gxy1) [label=left:${ G_{xy}^{[1]} \  }$, draw, circle, fill=black,minimum size=2pt, inner sep=0pt] {};

\draw (gxy1) -- ++(-90:2.0cm) node (zxy) [label=left:${ Z_{xy}\ \  }$, draw, circle, fill=black,minimum size=2pt, inner sep=0pt] {};

\draw (gy1) -- (gxy1);

\draw (gx1) -- ++(-90:2.0cm) node (gy2) [label=left:${ Z_x }$, draw, circle, fill=black,minimum size=2pt, inner sep=0pt] {};

\draw (gy1) -- ++(-90:2.0cm) node (gx2) [label = right:$ { Z_y }$, draw, circle, fill=black, minimum size=2pt,inner sep=0pt] {};
\draw(gx2)--(zxy);\draw(gy2)--(zxy);
\end{tikzpicture}
\end{center}

\begin{prop}
\label{prop:centraliser of zx}
The group $\cent{G_x}{Z_x}$ is intransitive on $\Gamma(x)$. Moreover, $Q_x$ is a Sylow 2-subgroup of $\cent{G_x}{Z_x}$.
\end{prop}
\begin{proof}
Since $Z_{xy}$ is a non-trivial subgroup of $G_x \cap G_y$ which is normalised by $\grp{\cent{G_x}{Z_x},G_e}$, the first part of the proposition holds by Lemma~\ref{lem:fundlem}(b). For the second part, since the action of $G_x$ on $\Gamma(x)$ is semiprimitive, we have that $\cent{G_x}{Z_x} \cap (G_x \cap G_y) = Q_x$. Now the claim follows since $G_x \cap G_y$ is a Sylow 2-subgroup of $G_x$.
\end{proof}

Now since $Z_x$ is non-trivial, there exists $v\in \Gamma$ such that $Z_x \not\leqslant G_v^{[1]}$. We set 
$$b=\min_{v \in \Gamma} \{\mathrm d(x,v) \mid Z_x \not\leqslant G_v^{[1]} \},$$
the parameter $b$ is called the \emph{critical distance}. Note that by vertex-transitivity, it suffices to ``measure'' the critical distance at $x$. Our goal is to bound $b$; we will see in Lemma~\ref{lem:boundonbisgood} that this enables us to bound the order of $G_x$. 

 The set of \emph{critical pairs} is 
$$\mathcal C = \{ (u,v) \mid \mathrm d(u,v) = b\text{ and } Z_u \not\leqslant G_v^{[1]} \}.$$
Examining the elements of $\mathcal C$ allows us to prove that a bound on $b$ exists.

\begin{prop}
If $(x,v) \in \mathcal C$ then $(v,x) \in \mathcal C$.
\end{prop}
\begin{proof}
 Suppose that $Z_v \leqslant  \vst{x}{1}$. Then Proposition \ref{p:zxcentral} implies that $[Z_v,Z_x]=1$.   Let $w$ be the unique vertex in $\Gamma(v)$ with $\mathrm d(x,w)=\mathrm d(x,v)-1$, by the minimality of $b$ we have $Z_x \leqslant G_w^{[1]}$ which gives $Z_x \leqslant  G_{vw} \leqslant  G_v$. Now $Z_x \leqslant \cent{G_v}{Z_v}$ which by Proposition  \ref{prop:centraliser of zx} has unique Sylow 2-subgroup $Q_v$. This implies $Z_x \leqslant  Q_v$, a contradiction to $(x,v) \in \mathcal C$.
\end{proof}

Fix critical pair $(x,v) \in \mathcal C$. We write the (unique) path between $x$ and $v$ as
$$(x,x+1,x+2,\dots,v-2,v-1,v).$$
We may assume that $y=x+1$ by arc-transitivity. The following diagram shows some inclusions between the subgroups we have defined so far.
\begin{figure}

\begin{center}
\begin{tikzpicture}[scale=0.8]

\draw (0,0) node (p12) [label=left:${\ \ G_x\cap G_{x+1}}$,draw,circle,fill=black,minimum size=2pt,
                            inner sep=0pt] {}
-- ++(135:2.0cm) node (p1) [label=above:$G_x$,draw,circle,fill=black,minimum size=2pt,
                            inner sep=0pt] {};
\draw (p12) -- ++(45:2.0cm) node (p2) [label=above:$G_{x+1}$,draw,circle,fill=black,minimum size=2pt,
                            inner sep=0pt] {};
\draw (p12) -- ++(-45:2.0cm) node (gy1) [label=above:${\ \ \ \ Q_{x+1} }$,draw,circle,fill=black,minimum size=2pt,
                            inner sep=0pt] {};
\draw (p12) -- ++(-135:2.0cm) node (gx1) [label=left:${ Q_x}$,draw,circle,fill=black,minimum size=2pt,
                            inner sep=0pt] {};
\draw (gx1) -- ++(-45:4.0cm) node (zx) [label=left:${ Z_x }$, draw, circle, fill=black,minimum size=2pt, inner sep=0pt] {};
\draw (gy1) -- ++(-45:4.0cm) node (zv) [label=right:${ Z_v }$, draw, circle, fill=black,minimum size=2pt, inner sep=0pt] {};
\draw (zx) -- ++(-45:2.0cm) node (zxv) [label=below:${ [Z_x,Z_v] }$, draw, circle, fill=black,minimum size=2pt, inner sep=0pt] {};
\draw (zxv) --  (zv)  ;
\draw (zx) -- ++(45:4.0cm) node (qv-1) [label=above:${ Q_{v-1}\ \ \ \  }$, draw, circle, fill=black,minimum size=2pt, inner sep=0pt] {};
\draw (zv) -- ++(45:4.0cm) node (qv) [label=right:${ Q_{v} }$, draw, circle, fill=black,minimum size=2pt, inner sep=0pt] {};
\draw (qv) -- ++(135:2.0cm) node (gvv-1) [label=right:${ G_{v} \cap G_{v-1} }$, draw, circle, fill=black,minimum size=2pt, inner sep=0pt] {};
\draw (gvv-1) -- ++(135:2.0cm) node (gv-1) [label=above:${  G_{v-1} }$, draw, circle, fill=black,minimum size=2pt, inner sep=0pt] {};
\draw (qv-1) -- ++(45:4.0cm) node (gv) [label=above:${  G_{v} }$, draw, circle, fill=black,minimum size=2pt, inner sep=0pt] {};
\end{tikzpicture}
\end{center}
\caption{Some inclusions of subgroups}
\label{figure}
\end{figure}
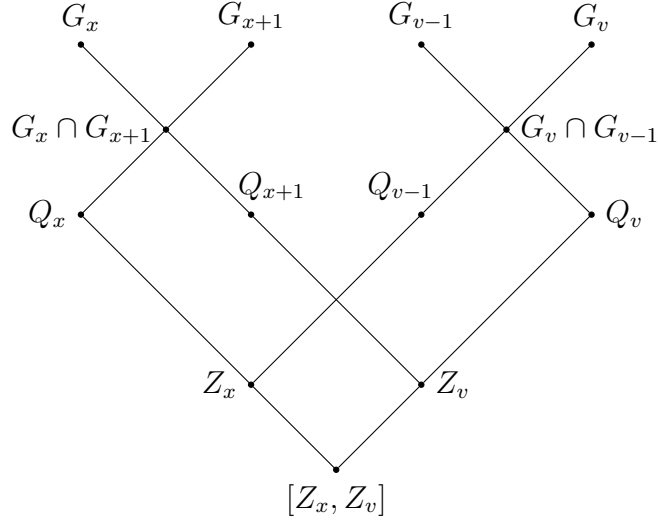
\begin{prop}
\label{prop:szzx/zxcapqv}
The following hold.
\begin{itemize}
\item[(i)] $|Z_x Q_v / Q_v  | = 2 = |Z_v Q_x / Q_x |$, $G_{v-1}\cap G_v=Z_xQ_v$ and $G_x \cap G_{x+1}=Z_vQ_x$.
\item[(ii)] $|Z_x / \cent{Z_x}{G_x \cap G_{x+1}}| =2$.
 \item[(iii)]  $Z_{xx+1}=\cent{Z_x}{G_x\cap G_{x+1}}$. 
\item[(iv)] $|Z_x / Z_{xx+1} | = 2=|Z_xZ_{x+1}/Z_x|$.
\end{itemize}
\end{prop}
\begin{proof}
Part (i) follows from $|(G_x \cap G_{x+1})/Q_x|=2$. For part (ii) we use part (i) to see that $\cent{Z_x}{G_x \cap G_{x+1}} = \cent{Z_x}{Z_vQ_x}=\cent{Z_x}{Z_v}$. Since $Z_x \leqslant Q_{v-1} \leqslant G_v$ we have $\cent{Z_x}{Z_v}= Z_x \cap \cent{G_v}{Z_v}$. By Proposition \ref{prop:centraliser of zx}, $Q_v$ is the unique Sylow 2-subgroup of $\cent{G_v}{Z_v}$, so we obtain $\cent{Z_x}{Z_v}= Z_x \cap Q_v$, and by  (i) we have $|Z_x / (Z_x \cap Q_v)| =2$.

 Clearly $Z_{xx+1}\leqslant  \cent{Z_x}{G_x\cap G_{x+1}}$. Since $Z_x\leqslant  G_x\cap G_{x+1}$ we have $\cent{Z_x}{G_x\cap G_{x+1}}=Z_x\cap\zent{G_x\cap G_{x+1}}$.     By Proposition \ref{p:zxcentral},   $Z_x$ is elementary abelian, and so $$Z_x \cap \zent{G_x \cap G_{x+1}} \leqslant \Omega_1(\zent{G_x \cap G_{x+1}}=Z_{xx+1}.$$ 
Part (iv) then follows from (ii) and (iii).
\end{proof}

\begin{lem}
\label{lem:chieffactors}
There is exactly one non-central $G_x$-chief factor in $Z_x$ and $\cent{Z_x}{G_x}=\cent{Z_x}{T}$ for $T\in\syl{3}{G_x}$.
\end{lem}
\begin{proof}
Let $m$ be the number of non-central $G_x$-chief factors in $Z_x$. Let $T\in \syl{3}{G_x}$.  Since $T$ is transitive on $\Gamma(x)$, Proposition \ref{prop:centraliser of zx} implies  that $[T,Z_x] \neq 1$ and since $T$ acts coprimely on $Z_x$, it follows that $m\geqslant 1$.   Since $TQ_x\vartriangleleft G_x$ (it has index two), we have that $\cent{Z_x}{TQ_x}=\cent{Z_x}{T}$ is $G_x$-invariant. Thus we may take the series $1 \leqslant  \cent{Z_x}{G_x} \leqslant  \cent{Z_x}{T} \leqslant  Z_x$ and refine it to a $G_x$-chief series.  Since the action is coprime $Z_x/\cent{Z_x}{T}$ is completely reducible as a $T$-module and contains no central $T$-chief factors, therefore no central $G_x$-chief factors. Note that by Proposition~\ref{p:Xmodules} each non-central $G_x$-chief factor has order $2^2$. On the other hand, by Proposition \ref{p:zxcentral}, $TQ_x$ centralises $\cent{Z_x}{T}$, and so $\cent{Z_x}{T}/\cent{Z_x}{G_x}$ is just a module for $G_x/TQ_x \cong \mathrm C_{2}$. Hence every $G_x$-chief factor in $\cent{Z_x}{T}$ is central. The following diagram therefore describes the structure of $Z_x$.
\begin{center}
\begin{tikzpicture}[scale=0.8]

\draw (0,0) node (zx) [label=right:${\ \ Z_{x}}$,draw,circle,fill=black,minimum size=2pt,
                            inner sep=0pt] {}
-- ++(-90:2.0cm) node (czxd) [label=right:${\ \ \cent{Z_x}{T}}$,draw,circle,fill=black,minimum size=2pt,
                            inner sep=0pt] {};
\draw (czxd) -- ++(-90:1.0cm) node (czxgx) [label=right:${\ \ \cent{Z_x}{G_x}}$,draw,circle,fill=black,minimum size=2pt,
                            inner sep=0pt] {};
\draw (czxgx) -- ++(-90:1.0cm) node (1) [label=right:${\ \ 1}$,draw,circle,fill=black,minimum size=2pt,
                            inner sep=0pt] {};
\draw [decorate,decoration={brace,amplitude=5pt},xshift=4pt,yshift=0pt] (-0.5,-2.0) -- (-0.5,0.0) node [black,midway,xshift=2.0pt] {$ 2^2/ \cdots / 2^2
\ \ \ \ \ \ \ \ \ 
\ \ \ \ \ \ \ \ \ 
\ \ \ \ \ \ \ \ \ 
$};
\draw [decorate,decoration={brace,amplitude=5pt},xshift=4pt,yshift=0pt] (-0.5,-4.0) -- (-0.5,-2.0) node [black,midway,xshift=2.0pt] {$2/\cdots/2
\ \ \ \ \ \ \ \ \ 
\ \ \ \ \ \ \ \ \ 
\ \ \ \ \ \ \ \ \ $};

\end{tikzpicture}
\end{center}
Set $Z_0=Z_x$ and $Z_m=\cent{Z_x}{T}$ and define the series 
$$Z_m \leqslant Z_{m-1} \leqslant \dots \leqslant Z_1 \leqslant Z_0$$
so that each $Z_i/Z_{i+1}$ is a $G_x$-chief factor.
 For $i\in [0,m)$, set $\ba{Z_i}=Z_i/Z_{i+1}$. Since $G_x \cap G_{x+1}$ is a 2-group, $\cent{\ba{Z_i}}{G_x \cap G_{x+1}} \neq 1$. On the other hand, $G_x \cap G_{x+1}$ cannot centralise $\ba{Z_i}$ since $\cent{G_x}{\ba{Z_i}}$ is a normal subgroup of $G_x$ containing $Q_x$, so $G_x \cap G_{x+1} \leqslant \cent{G_x}{\ba{Z_i}}$ implies that $G_x=\cent{G_x}{\ba{Z_i}}$ which is against $\ba{Z_i}$ being a non-central $G_x$-chief factor. Hence for $0\leqslant i \leqslant m-1$ we have $| \ba{Z_i} /\cent{\ba{Z_i}}{G_x\cap G_{x+1}}|=2$. Now $|Z_x : \cent{Z_x}{G_x \cap G_{x+1}}|=2$ by Proposition \ref{prop:szzx/zxcapqv} (ii). So 
(by \cite[Lemma 2.21]{chrispeter}) we have that 
$$2=|Z_x / \cent{Z_x}{G_x \cap G_{x+1}}| \geqslant  \prod_{i=0}^{m-1} |\ba{Z_i}/\cent{\ba{Z_i}}{G_x \cap G_{x+1}}| = 2^m$$ which gives $m\leqslant 1$. Since $G_x=T(G_x\cap G_{x+1})$ we have 
$$\cent{\cent{Z_x}{T}}{G_x \cap G_{x+1}} = \cent{Z_x}{G_x}.$$ 
Now we see that
$$2 \geqslant   |\ba{Z_0}/\cent{\ba{Z_0}}{G_x \cap G_{x+1}}||\cent{Z_x}{T}/\cent{\cent{Z_x}{T}}{G_x \cap G_{x+1}}| = 2|\cent{Z_x}{T}/\cent{Z_x}{G_x}|$$
which gives the result.
\end{proof}

\begin{lem}
\label{lem:fixedpointsof3subgps}
Let $T\in\syl{3}{G_x}$ and let $T_0 \leqslant T$ such that $ \cent{Z_x}{T} < \cent{Z_x}{T_0}$. Then $[T_0,Z_x]=1$.
\end{lem}
\begin{proof}
Set $W=\cent{Z_x}{T_0}$ and suppose that $ \cent{Z_x}{T} < W$. Now $\cent{Z_x}{T}< W \leqslant  Z_x$. Since $|Z_x/\cent{Z_x}{T}|=2^2$, by Lemma \ref{lem:chieffactors} and Proposition \ref{p:Xmodules}, we see that either $W=Z_x$ (and we are done) or  $|Z_x:W|=2$.  Since $T_0$ acts on $Z_x/W$ it must centralise $Z_x/W$, which implies $[Z_x,T_0] \leqslant W=\cent{Z_x}{T_0}$. Now coprime action gives 
$$[Z_x,T_0]=[Z_x,T_0,T_0] \leqslant [\cent{Z_x}{T_0},T_0] =1.$$
\end{proof}

\begin{rem}
\label{rem:remark}
Let $r$ be a vertex of $\Gamma$. For each $t \in \Gamma(r)$ we may identify $\Gamma(r)$ with $V$ so that $t$ is identified with the zero vector.  For $s \in \Gamma(r)$ distinct from $t$, we
 see that $\grp{G_r \cap G_t, G_r \cap G_s}/Q_r$ is generated by two involutions, and so is isomorphic to $\sym{3}$. Therefore, by Proposition~\ref{genofx}, for every basis $\mathcal B$  of $V$ we have
 $$\grp{G_r \cap G_t, G_r \cap G_s \mid s \in \mathcal B} = G_r.$$ 
\end{rem}

%


For $(r,s) \in \mathcal C$ with the (unique) path between $r$ and $s$ being $(r,r+1,\dots,s-1,s)$ we set
\begin{eqnarray*}
\mathcal C_{-1}(r,s) & = & \{ y \in \Gamma(r) \setminus \{r+1\} \mid (y,s-1) \in \mathcal C\}, \\
\mathcal C_{-1}^*(r,s) & = &\{ y \in \Gamma(r) \setminus \{r+1\} \mid (y,s-1) \notin \mathcal C\}, \end{eqnarray*}
and note that $\Gamma(r) = \{r+1\} \cup \mathcal C_{-1}(r,s) \cup \mathcal C_{-1}^*(r,s)$. The composition of the defined sets plays a prominent role in the proof of the next theorem.

\begin{thm}
\label{b <3}
We have $b\leqslant 2$.
\end{thm}
\begin{proof}
We assume for a contradiction that $b>2$. Figure~\ref{figure} will be helpful in recognising various inclusions of subgroups throughout the proof.

The assumption implies that $Z_x \leqslant Q_{x+1}$, so in particular, $Z_x$ normalises $Z_{x+1}$ (and vice versa) and $Z_xZ_{x+1}$ is a subgroup of $G_x \cap G_{x+1}$ normalised by $G_e$.

\begin{claim}
$Z_v$ fixes setwise $\mathcal C_{-1}(x,v)$ and $\mathcal C_{-1}^*(x,v)$.
\end{claim}

Since $Z_v \leqslant G_x \cap G_{x+1}$ and $Z_v \not\leqslant Q_x$, certainly $Z_v$ induces an involution on the set $\Gamma(x) \setminus \{x+1\}$. Now since $Z_v \leqslant Q_v \leqslant G_v \cap G_{v-1}$ normalises $Q_{v-1}$, for any $a\in Z_v$ and $t\in \Gamma(x) \setminus \{x+1\}$ we have that $Z_t \leqslant Q_{v-1}$ if and only if $(Z_t)^a=Z_{t^a}\leqslant Q_{v-1}$, that is, $t \in \mathcal C_{-1}^*(x,v)$ if and only if $t^a \in \mathcal C_{-1}^*(x,v)$.\\ \\

\begin{claim} Suppose $t\in \mathcal C_{-1}^*(x,v)$. Then $Z_tZ_x = Z_xZ_{x+1}$.
\end{claim}

We have $Z_{t} \leqslant Q_{v-1} \leqslant G_{v-1} \cap G_v = Z_x Q_v \leqslant Z_x \cent{G_v}{Z_v}$ by Propositions \ref{prop:szzx/zxcapqv}(i) and \ref{prop:centraliser of zx}. Thus  $[Z_{t},Z_v] \leqslant [Z_x,Z_v]$
and since $Z_v\leqslant  G_x$ it follows that $[Z_x,Z_v] \leqslant Z_x \leqslant Z_tZ_x$, so $Z_v$ normalises $Z_{t}Z_x$. Hence $Z_v Q_x = G_x \cap G_{x+1}$ (by Proposition \ref{prop:szzx/zxcapqv}(i)) normalises $Z_{t} Z_x$ and so does $G_{t} \cap G_x$. Since 
$$ H:= \grp{ G_{t} \cap G_x , G_x \cap G_{x+1} }$$
normalises $Z_{t}Z_x$ and induces a $\sym{3}$ on $\Gamma(x)$ which contains an element $h$ interchanging $t$ and $x+1$, we have $Z_{t}Z_x = (Z_{t}Z_x)^h = Z_{x+1}Z_x$.\\ \\


For the next two claims we use Remark~\ref{rem:remark}, identifying $\Gamma(x)$ with $V$ so that $x+1$ is identified with the zero vector.

\begin{claim}
\label{enough crit pairs}
There are no bases contained in $  \mathcal C_{-1}^*(x,v)$.
\end{claim}

Suppose that $\mathcal B$ is a basis of $V$ contained in $\mathcal C_{-1}^*(x,v)$. We have  $$G_x = \grp{G_x \cap G_{x+1}, G_r \cap G_x \mid r \in \mathcal B}.$$ For each $r \in \mathcal B$ the  previous claim implies
$Z_{x+1}Z_x = Z_x Z_r$
and so  $1 \neq Z_{x+1}Z_x $ is normalised by $\grp{G_x,G_e}$, a contradiction to Lemma~\ref{lem:fundlem}.\\ \\

\begin{claim}
\label{wherecritpairs}
$\mathcal C_{-1}(x,v)$ contains a basis.
\end{claim}

Let $W$ be the subspace of $V$ generated by the vectors in $\mathcal C_{-1}^*(x,v)$. The previous claim shows that $W \neq V$, so  $V$ has a basis $\mathcal B$ such that $\mathcal B \cap W = \varnothing$. Therefore $\mathcal B \cap \mathcal C_{-1}^*(x,v)=\varnothing$ and since $\mathcal B \subseteq V^\#$ we have that $\mathcal B \subseteq \mathcal C_{-1}(x,v)$ as required.\\ \\

For $(x,v) \in \mathcal C$ we write $\Pi_{x,v}$ for the basis delivered by \ref{wherecritpairs}. For $\alpha \in \Pi_{x,v}$ set 
 $$R_\alpha := [Z_\alpha,Z_{v-1}].$$ 
 Since $Q_\alpha$ is the unique Sylow 2-subgroup of $\cent{G_\alpha}{Z_\alpha}$, by Proposition \ref{prop:centraliser of zx}, and $Z_{v-1}\not\leqslant  Q_\alpha$, as $(\alpha,v-1)$ is a critical pair, we have that $R_\alpha\neq 1$. Note that $R_\alpha \leq Z_\alpha \cap Z_{v-1}$ and is therefore centralised by $Q_\alpha$ and $Q_{v-1}$. 
  The figure below indicates   the position of the vertices we have defined.

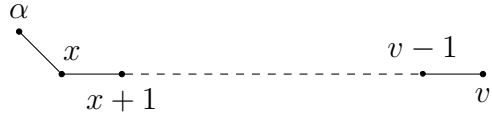
\begin{center}
\begin{tikzpicture}[scale=0.8]

\draw (0,0) node (x) [label=above:${\ \ x}$,draw,circle,fill=black,minimum size=2pt,
                            inner sep=0pt] {}
-- ++(1,0) node (y) [label=below:${x+1}$,draw,circle,fill=black,minimum size=2pt,
                            inner sep=0pt] {};
\draw [dashed] (y) -- ++(5,0) node (v-1) [label=above:${v-1}$,draw,circle,fill=black,minimum size=2pt,
                            inner sep=0pt] {};
\draw (v-1) -- ++(1,0) node (v) [label=below:${v}$,draw,circle,fill=black,minimum size=2pt, inner sep=0pt] {};
\draw (x) -- ++(135:1) node (x-1) [label=above:${\alpha}$,draw,circle,fill=black,minimum size=2pt,inner sep=0pt] {};

\end{tikzpicture}
\end{center}

\begin{claim}
For $\alpha \in \Pi_{x,v}$ we have $R_\alpha \leqslant Z_x$.
\end{claim}

Since $(\alpha,v-1)$ is a critical pair, Proposition \ref{prop:szzx/zxcapqv}(i) yields $G_\alpha \cap G_x = Z_{v-1} Q_{\alpha}$ so 
$$R_\alpha \leqslant \cent{Z_{\alpha}}{G_{\alpha} \cap G_x} = Z_{x \alpha} \leqslant Z_x$$
by Proposition \ref{prop:szzx/zxcapqv}(iii).\\ \\ 

\begin{claim}
\label{commutators in centres}
There exists $\alpha \in \Pi_{x,v}$ such that $R_\alpha \leqslant \cent{Z_x}{G_x}$.
\end{claim}

Suppose this is not the case and that for each $\alpha \in \Pi_{x,v}$ we have $R_\alpha \not\leqslant \cent{Z_x}{G_x}$. Since $(\alpha,v-1)$ is a critical pair, we have that $G_{\alpha} \cap G_x = Q_{\alpha} Z_{v-1}$ centralises $R_\alpha$ (since $R_{\alpha} \leqslant Z_\alpha \cap Z_{v-1}$). Since $b>1$ we have that $Z_v \leqslant Q_{v-1}\leqslant  \cent{G_{v-1}}{Z_{v-1}}$,  by Proposition \ref{prop:centraliser of zx}. Then as $R_\alpha\leqslant  Z_{v-1}$ it follows that $R_\alpha$ is centralised by $Z_v$, and therefore by $Z_vQ_x=G_x \cap G_{x+1}$ (since \ref{commutators in centres} shows $R_\alpha \leqslant Z_x$). Thus $R_\alpha$ is centralised by $C_\alpha:=\grp{G_x \cap G_{x+1},G_x \cap G_\alpha}$.

Pick $D_\alpha \in \syl{3}{C_\alpha}$ (so $D_\alpha \cong \mathrm C_{3}$). Since $n\geqslant 2$, Lemma \ref{lem:fixedpointsof3subgps} implies that $[Q_x D_\alpha,Z_x]=1$. Now $C:=\grp{Q_x D_\alpha \mid \alpha \in \Pi_{x,v}}$ centralises $Z_x$, and by our definition of $\Pi_{x,v}$ we see $C$ contains a Sylow 3-subgroup of $G_x$, and so transitive on $\Gamma(x)$, contradicting Proposition \ref{prop:centraliser of zx}. \\ \\

For each  $\alpha \in \Pi_{x,v}$  we apply \ref{wherecritpairs} to $\mathcal C_{-1}(\alpha,v-1)$ which gives $\Pi_{\alpha,v-1}$. Then we apply \ref{commutators in centres} to $\Pi_{\alpha,v-1}$ which yields $\sigma \in \mathcal C_{-1}(\alpha,v-1)$ such that 
$$1\neq R_\sigma = [Z_\sigma,Z_{v-2}] \leqslant \cent{Z_\alpha}{G_\alpha}$$
and we let $\Sigma$ be the set of vertices obtained. We indicated below where the vertices in $\Sigma$ lie.

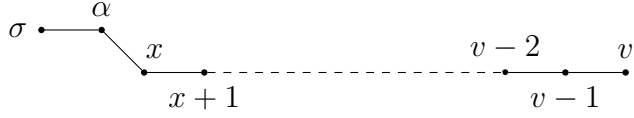
\begin{center}
\begin{tikzpicture}[scale=0.8]

\draw (0,0) node (x) [label=above:${\ \ x}$,draw,circle,fill=black,minimum size=2pt,
                            inner sep=0pt] {}
-- ++(1,0) node (y) [label=below:${x+1}$,draw,circle,fill=black,minimum size=2pt,
                            inner sep=0pt] {};
\draw [dashed] (y) -- ++(5,0) node (v-2) [label=above:${v-2}$,draw,circle,fill=black,minimum size=2pt,
                            inner sep=0pt] {};
\draw (v-2) -- ++(1,0) node (v-1) [label=below:${v-1}$,draw,circle,fill=black,minimum size=2pt, inner sep=0pt] {};
\draw (v-1) -- ++(1,0) node (v) [label=above:${v}$,draw,circle,fill=black,minimum size=2pt, inner sep=0pt] {};
\draw (x) -- ++(135:1) node (x-1) [label=above:${\alpha}$,draw,circle,fill=black,minimum size=2pt,inner sep=0pt] {};

\draw (x-1) -- ++(180:1) node (s) [label=left:${\sigma}$,draw,circle,fill=black,minimum size=2pt,inner sep=0pt] {};
\end{tikzpicture}
\end{center}

\begin{claim}\label{claim:rsig in zx} For all $\sigma \in \Sigma$ we have $R_\sigma  \leqslant Z_x$.\end{claim}

We have that $R_\sigma \leqslant \cent{Z_\sigma}{Q_\sigma Z_{v-2}}=Z_\sigma \cap Z_\alpha$, by Proposition \ref{prop:szzx/zxcapqv}(i) and (iii). Now since $b>1$ we have that $Z_{v-1} \leqslant Q_{v-2}$ and so $[Z_{v-1},R_\sigma] \leqslant [Q_{v-2},Z_{v-2}]=1$, by Proposition \ref{prop:centraliser of zx}. Hence $R_\sigma \leqslant \cent{Z_\alpha}{Q_\alpha Z_{v-1} }= \cent{Z_\alpha}{G_\alpha \cap G_x}= Z_\alpha \cap Z_x$, by Proposition \ref{prop:szzx/zxcapqv}(i) and (iii).\\ \\

\begin{claim} For all $\sigma \in \Sigma$ we have $[Z_v,R_\sigma]=1$. \end{claim}

The assumption $b>2$ implies that $Z_v \leqslant Q_{v-2}$, and therefore $Z_v$ centralises $Z_{v-2}$ which contains $R_\sigma$.\\ \\ 

\begin{claim} \label{claim: sig cent} There is some $\sigma\in \Sigma$ such that $R_\sigma$ is centralised by $G_x$.\end{claim}

Assume the claim is false. Then for all $\sigma \in \Sigma$ by \ref{claim:rsig in zx} we have $R_\sigma \not\leqslant \cent{Z_x}{G_x}$. Let $\sigma \in \Sigma$ be arbitrary.  By \ref{claim:rsig in zx} and Proposition \ref{prop:centraliser of zx},
we see that $R_\sigma$ is centralised by $Q_x$. 
Using \ref{claim: sig cent} we see that $R_\sigma$ is centralised by $C_\sigma:=\grp{Q_xZ_v,G_x \cap G_{\alpha}}$. Pick $D_\sigma \in \syl{3}{C_\sigma}$. Then $|D_\sigma|=3$ and since $n\geqslant 2$, Lemma \ref{lem:fixedpointsof3subgps} implies $[D_\sigma,Z_x]=1$. Now $D:=\grp{Q_xD_\sigma \mid \sigma \in \Sigma}$ centralises $Z_x$ and by the definition of $\Pi_{x,v}$ we see that $D$ contains a Sylow 3-subgroup of $G_x$. Therefore $D$ is a transitive subgroup of $G_x$,  a contradiction to Proposition \ref{prop:centraliser of zx}. \\ \\

\begin{claim} A contradiction. \end{claim}

By \ref{claim: sig cent} there is $\sigma \in \Sigma$ and $\alpha \in \Pi_{x,v}$ such that $1 \neq R_\sigma \leqslant \cent{Z_\alpha}{G_\alpha} \cap \cent{Z_x}{G_x}$. However,  $\cent{Z_\alpha}{G_\alpha} \cap \cent{Z_x}{G_x}$ is centralised by $\grp{G_{\alpha},G_x}$ and contained in $G_{\alpha} \cap G_x$,  so Lemma~\ref{lem:fundlem} yields
$$\cent{Z_\alpha}{G_\alpha} \cap \cent{Z_x}{G_x}= 1,$$ a contradiction which completes the proof.
\end{proof}

Theorem~\ref{mainthm} now follows from Theorem~\ref{b <3} and the lemma below.

\begin{lem}
\label{lem:boundonbisgood}
The group $\vst{x}{b+2}$ is trivial for all $x\in \Gamma$.
\end{lem}
\begin{proof}
Let $x$ be an arbitrary vertex of $\Gamma$ and choose $v$ so that $(x,v) \in \mathcal C$. As in Remark~\ref{rem:remark} we identify the set $\Gamma(v)$ with $V$ so that $v-1$ is identified with {\underline 0}. For each $w\in V^\#$  we have $\vst{w}{b+2} \leqslant \vstx{1}=Q_x$, so by Proposition \ref{prop:centraliser of zx}, $[Z_x,\vst{w}{b+2}]=1$. Now $(x,v) \in \mathcal C$ implies that $Z_x$ induces an involution on $\Gamma(v)$ (which fixes {\underline 0}). We have therefore that the action of $Z_x$ on $V^\#$ is the map
$$w \mapsto -w.$$
Since $Z_x$ centralises $\vst{w}{b+2}$, for $w\in V^\#$ we have 
$$\vst{w}{b+2}=\vst{-w}{b+2}.$$
For $w\in V^\#$ the set $ B_w:=\{{\underline 0},w,-w\}$ is a block for the action of $G_v$ on $\Gamma(v)$ and $G_{ B_w}$ is 2-transitive on $ B_w$. This implies that 
$$\vst{w}{b+2}=\vst{-w}{b+2}=\vst{{\underline 0}}{b+2}.$$
 Since this is true for all $w\in V^\#$, $\vst{{\underline 0}}{b+2}$ is normalised by $G_v$. Now we see that 
$$\vst{{\underline 0}}{b+2} \vartriangleleft \grp{G_v,G_{{\underline 0}}}$$
hence $\vst{{\underline 0}}{b+2}=1$ by Lemma~\ref{lem:fundlem}(a).
\end{proof}

\begin{rem}
The proof shows that $G_x^{[4]}=1$, we would like to know if there exist examples with $G_x^{[3]} \neq 1$. The largest examples of which the authors are aware have $|G_x|=144$, $|G_x^{[1]}| =8$, $|G_x^{[2]}|=2$ and $|G_x^{[3]}|=1$.
\end{rem}


\begin{thebibliography}{99}

\bibitem{BM}
\'A. Bereczky, A. Mar\'oti, On groups with every normal subgroup transitive or semiregular, J. Algebra 319 (2008) 1733--1751.

\bibitem{stelldelgado}
A. Delgado, D. Goldschmidt, B. Stellmacher, Groups and graphs: new results and methods. DMV Seminar, 6. Birkh\"auser Verlag, Basel, 1985.

\bibitem{gard}
A. Gardiner,  Arc-transitivity in graphs, Quart. J. Math. Oxford 24 (1973) 399--407.

\bibitem{gold}
D. M. Goldschmidt, Automorphisms of Trivalent Graphs,  Ann. of Math. (2) 111 (1980), no. 2, 377--407.

\bibitem{kurzweilstellmacher}
H. Kurzweil,  B. Stellmacher, The theory of finite groups. 
An introduction.  Universitext. Springer-Verlag, New York, 2004.


\bibitem{chrispeter}
C. W. Parker, P. Rowley, Symplectic amalgams. Springer Monographs in Mathematics. Springer-Verlag London, Ltd., London, 2002. 

\bibitem{praeger}
C.E. Praeger, Finite quasiprimitive group actions on graphs and designs, in: Young Gheel Baik, David L. Johnson, Ann Chi
Kim (Eds.), Groups -- Korea, de Gruyter, Berlin, New York, 2000, pp. 319--331.

\bibitem{PPSS}
C. E. Praeger, L. Pyber, P. Spiga, E. Szab{\'o}, Graphs with automorphism groups admitting composition factors of bounded rank. Proc. Amer. Math. Soc. 140 (2012), no. 7, 2307--2318.

\bibitem{PrSV}
C. E. Praeger,  P. Spiga, G. Verret, Bounding the size of a vertex-stabiliser in a finite vertex-transitive graph. J. Combin. Theory Ser. B 102 (2012), no. 3, 797--819.

\bibitem{PSV}
P. Poto\v{c}nik, P. Spiga, G. Verret, On graph-restrictive permutation groups. J. Combin. Theory Ser. B 102 (2012), no. 3, 820--831.


\bibitem{sami}
A. Q. Sami, Locally dihedral amalgams of odd type. J. Algebra 298 (2006) 630--644.

\bibitem{serre}
J-P. Serre, Trees. Translated from the French original by John Stillwell.  Springer Monographs in Mathematics. Springer-Verlag, Berlin, 2003.

\bibitem{spigath-w}
P. Spiga, Two local conditions on the vertex stabiliser of arc-transitive graphs and their effect on the Sylow subgroups. J. Group Theory 15 (2012), no. 1, 23--35.

\bibitem{pabloverret}
P. Spiga, G. Verret, On intransitive graph-restrictive groups. to appear in J. Algebraic Combin.  doi:10.1007/s10801-013-0482-5.



\bibitem{trof1}
V. I. Trofimov, Stabilizers of the vertices of graphs with projective suborbits. Soviet Math. Dokl., 42 (1991), 825--828.

\bibitem{trof2}
V. I. Trofimov, Graphs with projective suborbits, Russian Acad. Sci. Izv. Math., 39 (1992), 869--894.

\bibitem{trof3}
V. I. Trofimov, Graphs with projective suborbits. Cases of small characteristics. I. Russian Acad. Sci. Izv. Math., 45 (1995), 353--398. 

\bibitem{trof4}
V. I. Trofimov, Graphs with projective suborbits. Cases of small characteristics. II. Russian Acad. Sci. Izv. Math., 45 (1995), 559--576.

\bibitem{trof5}
V. I. Trofimov, Graphs with projective suborbits. Exceptional cases of characteristic 2. I. Izv. Math., 62 (1998), 1221--1279.

\bibitem{trof6}
V. I. Trofimov, Graphs with projective suborbits. Exceptional cases of characteristic 2. II. Izv. Math., 64 (2000), 173--192.

\bibitem{trof7}
V. I. Trofimov, Graphs with projective suborbits. Exceptional cases of characteristic 2. III. Izv. Math., 65 (2001), 787--822.

\bibitem{trof8}
V. I. Trofimov, Graphs with projective suborbits. Exceptional cases of characteristic 2. IV. Izv. Math., 67 (2003), 1267--1294.

\bibitem{trofweiss1}
V. I. Trofimov, R. M. Weiss, Graphs with a locally linear group of automorphisms. Math. Proc. Cambridge Philos. Soc. 118 (1995), no. 2, 191--206. 

\bibitem{trofweiss2}
V. I. Trofimov, R. M. Weiss, The group {$E_6(q)$} and graphs with a locally linear group of automorphisms. Math. Proc. Cambridge Philos. Soc. 148 (2010), no. 1, 1--32. 

\bibitem{Tutte1}
W. T. Tutte, A family of cubical graphs. Proc. Cambridge Philos. Soc. 43 (1947), 459--474.


\bibitem{Tutte2}
W.~T.~Tutte, On the symmetry of cubic graphs. Canad. J. Math. 11 (1959), 621--624.

\bibitem{verret}
G. Verret, On the order of arc-stabilisers in arc-transitive graphs. Bull. Aust. Math. Soc. 80 (2009), 498--505.

\bibitem{weisscon}
R. Weiss, s-transitive graphs. Colloq. Math. Soc. J\'anos Bolyai 25 (1978) 827--847.

\bibitem{weissgold}
R. Weiss, On a theorem of Goldschmidt. Ann. of Math. (2) 126 (1987), no. 2, 429--438. 

\end{thebibliography}
\end{document}